\documentclass[10pt,a4paper,reqno]{amsart}
\usepackage{fullpage}
\usepackage{hyperref}
\numberwithin{equation}{section}
\usepackage{comment}

{\theoremstyle{plain}
\newtheorem{theorem}{Theorem}[section]
\newtheorem{lemma}[theorem]{Lemma}
\newtheorem{corollary}[theorem]{Corollary}}

{\theoremstyle{remark}   
\newtheorem{example}[theorem]{Example}}

\title{On integral mixed Cayley graphs over non-abelian finite groups admitting an abelian subgroup of index 2}

\author{Angelot Behajaina}
\email{angelot.behajaina@universite-paris-saclay.fr}
\address{Universit\'e Paris-Saclay, CNRS, Laboratoire de Math\'ematiques d'Orsay, 91405 Orsay, France}

\author{Fran\c cois Legrand}
\email{francois.legrand@unicaen.fr}
\address{Normandie Univ., UNICAEN, CNRS, Laboratoire de Math\'ematiques Nicolas Oresme, 14000 Caen, France}



\begin{document}

\maketitle

\vspace{-12mm}

\begin{abstract}
Recently, several works by a number of authors have provided characterizations of integral undirected Cayley graphs over generalized dihedral groups and generalized dicyclic groups. We generalize and unify these results in two different ways. Firstly, we work over arbitrary non-abelian finite groups admitting an abelian subgroup of index 2. Secondly, our main result actually characterizes integral mixed Cayley graphs over such finite groups, in the spirit of a very recent result of Kadyan--Bhattarcharjya in the abelian case.
\end{abstract}

\section{Introduction} \label{sec:intro}

In this note, we consider {\it{mixed graphs}} with no loop and no multiple edge, i.e., couples $(V,E)$, where $V$ is a non-empty finite set and $E$ is a subset of $(V \times V) \setminus \{(v,v) \, : \, v \in V\}$ (see \S\ref{sec:prelim} for the basic terminology used in the sequel). A mixed graph is {\it{integral}} if the eigenvalues of its Hermitian adjacency matrix, a matrix introduced independently by Liu--Li and Guo--Mohar in \cite{LL15} and \cite{GM17}, respectively, are all in $\mathbb{Z}$. The aim of the present note is to contribute to the study of the integrality of a special class of mixed graphs, namely mixed Cayley graphs. For a subset $S$ of a finite group $G$ with $1 \not \in S$, by the {\it{mixed Cayley graph}} over $G$ with respect to $S$, we mean the mixed graph $(G, \{(g,h) \in G^2 : g^{-1} h \in S\})$, denoted by ${\rm{Cay}}(G,S)$.

Motivated by works of Harary--Schwenk \cite{HS74} and {Ahmady--Alon--Blake--Shparlinski} \cite{AABS09}, many authors studied the integrality of {\it{undirected Cayley graphs}}, i.e., of mixed Cayley graphs ${\rm{Cay}}(G,S)$ which assume $S^{-1}=S$ (for results on the integrality of other kinds of mixed graphs, see, e.g., \cite{BC76, Wat79, WS79} and, for a survey of some of the known results on eigenvalues of mixed Cayley graphs, see \cite{LZ22}). For $G$ abelian, se\-veral characterizations of integrality were obtained by Bridges--Mena \cite{BM82}, Klotz--Sander \cite{KS10}, and Alperin--Peterson \cite{AP12}. Moreover, Lu--Huang--Huang focused on the case where $G$ is dihedral in \cite{LHH18}, and their results were extended to genera\-lized dihedral groups by Huang--Li in \cite{HL21}. Furthermore, Cheng--Feng--Huang considered in \cite{CFH19} the case where $G$ is a dicyclic group, and their results were extended to generalized dicyclic groups in our previous work \cite{BL22}. 

However, the integrality of mixed Cayley graphs which are not necessarily undirected seems to be a very recent topic, and we are aware of only one result of Kadyan--Bhattarcharjya \cite{KB21}, who characterized integral mixed Cayley graphs over abelian groups. Let us say that their characterization was extended by Huang--Lu--M\"onius in \cite{HLM22}, who determined the splitting field of the characteristic polynomial of the Hermitian adjacency matrix of any mixed Cayley graph in the abelian case.

Such characterizations of integrality usually involve the classification of the (complex) irreducible re\-pre\-sentations of the underlying finite group $G$ and, via a result of Babai \cite{Bab79}, are derived from computing the eigenvalues of some endomorphisms of some vector spaces whose dimensions are those of the irreducible representations of $G$. After abelian groups, it is then natural to study the case where $G$ is non-abelian and where every irreducible representation of $G$ has dimension at most 2. By a result of Amitsur \cite[Theorem 3]{Ami61}, the latter can happen only in the next two situations: {\rm{(1)}} $G$ has an abelian subgroup of index 2, and {\rm{(2)}} $G/Z(G) \cong (\mathbb{Z}/2\mathbb{Z})^3$.

Here we work over arbitrary non-abelian finite groups $G$ as in (1), and our main result (The\-o\-rem \ref{thm:main}) characterizes integral mixed Cayley graphs over such finite groups $G$, in terms of the irreducible representations of the underlying finite group. In particular, Theorem \ref{thm:main} generalizes and unifies various results from the works \cite{LHH18, CFH19, HL21, BL22} quoted above in two different ways. As a first step, we classify all ir\-re\-ducible representations (up to equivalence) of any of our finite groups $G$ (see \S\ref{sssec:main_1}).

Theorem \ref{thm:main} has several consequences. In the undirected case, our main result takes a simple form (see Corollary \ref{coro:undirected}) and we derive a convenient sufficient condition, involving the Boolean algebra of the underlying abelian subgroups, for an undirected Cayley graph over any of our groups $G$ to be integral (see Corollary \ref{coro:simple}). But Theorem \ref{thm:main} also applies to non-necessarily undirected Cayley graphs. For example, in the case of generalized dihedral/dicyclic groups, our main result also takes a simple form (see Corollary \ref{coro:s=-1}) and, in particular, it yields the first integrality results for non-necessarily undirected Cayley graphs in the non-abelian case (see also Corollaries \ref{coro:1} and \ref{coro:2} for results devoted to the directed case). Our results are illustrated by several explicit examples (see Examples \ref{ex:1}, \ref{ex:1.5} and \ref{ex:2}).

\section{Preliminaries} \label{sec:prelim}

We collect the basic material on representations of finite groups and graphs that will be used in the sequel. For a subset $S$ of a finite group $G$ and a non-zero integer $n$, we set $S^n = \{s^n \, : \, s \in S\}$.

\subsection{Representations of finite groups} \label{ssec:prelim_1}

A (complex) {\it{representation}} of a given finite group $G$ is a group homomorphism $\rho : G \rightarrow {\rm{GL}}(V)$, where $V$ is a finite dimensional complex vector space. The {\it{dimension}} ${\rm{dim}}(\rho)$ of $\rho$ is the dimension of $V$ and the {\it{character}} of $\rho$ is the map $\chi_\rho : G \rightarrow \mathbb{C}$ defined by $\chi_\rho(g) = {\rm{Tr}}(\rho(g))$ for $g \in G$.

A subspace $W$ of $V$ is {\it{invariant}} under $\rho$ if $\rho(g)(W) = W$ for $g \in G$, and $\rho$ is {\it{irreducible}} if $V \not=\{0\}$ and if $\{0\},V$ are the only invariant subspaces of $V$. Letting $L^2(G)$ be the vector space of all functions $f : G \rightarrow \mathbb{C}$, together with the Hermitian inner product $( \cdot | \cdot)$ given by
\begin{equation} \label{eq:scalar}
(f_1 | f_2) = \frac{1}{|G|} \sum_{g \in G} f_1(g) \overline{f_2(g)}
\end{equation}
for $f_1, f_2 \in L^2(G)$, the representation $\rho$ is irreducible if and only if $(\chi_\rho| \chi_\rho)=1$. 

Two representations $\rho_1 : G \rightarrow {\rm{GL}}(V_1)$ and $\rho_2 : G \rightarrow {\rm{GL}}(V_2)$ of $G$ are {\it{equivalent}} if there is an isomorphism $T : V_1 \rightarrow V_2$ such that, for every $g \in G$, we have $T \circ \rho_1(g) = \rho_2(g) \circ T$. Classically, $\rho_1$ and $\rho_2$ are equivalent if and only if $\chi_{\rho_1} = \chi_{\rho_2}$. The group $G$ has only finitely many inequivalent irreducible representations, and their number equals the number of conjugacy classes of $G$. Moreover, if $\rho_1, \dots, \rho_n$ are the inequivalent irreducible representations of $G$, then 
\begin{equation} \label{eq:squares}
|G| = ({\rm{dim}}(\rho_1))^2 + \cdots + ({\rm{dim}}(\rho_n))^2.  
\end{equation}
In particular, if $G$ is abelian, then every irreducible representation of $G$ has dimension 1.

For a finite set $S$ and a function $f : S \rightarrow \mathbb{C}$, we set 
$$f(S) = \sum_{s \in S} f(s).$$ 
A subset $S$ of $G$ is {\it{integral}} if $\chi_\rho(S) \in \mathbb{Z}$ for every irreducible representation $\rho$ of $G$. For $G$ abelian, we shall also need the characterization of integral subsets of $G$ in terms of the Boolean algebra of $G$. To that end, we let $\mathcal{F}_G$ be the set of all subgroups of $G$. The {\it{Boolean algebra}} $\mathbb{B}(G)$ is the set whose elements are obtained by arbitrary finite intersections, unions, and complements of elements of $\mathcal{F}_G$. The next lemma has been established by Alperin--Peterson in \cite{AP12}:

\begin{lemma} \label{lemma:ap}
A subset $S$ of a finite abelian group $G$ is integral if and only if $S \in \mathbb{B}(G)$.
\end{lemma}

\noindent
Let us also recall that, for $G$ abelian, the minimal elements of $\mathbb{B}(G)$ are called {\it{atoms}}. As proved by Alperin--Peterson \cite{AP12}, every element of $\mathbb{B}(G)$ is the union of some atoms of $G$, and every atom is of the form $\{ x \in G \, : \, \langle x \rangle = \langle g \rangle \}$ ($g \in G$). In particular, we have:

\begin{lemma} \label{lemma:ap_0}
Let $G$ be a finite abelian group and $S \in \mathbb{B}(G)$. Then $S = S^j$ for every non-zero integer $j$ which is coprime to every element order in $G$.
\end{lemma}

\subsection{Graphs} \label{ssec:prelim_2}

A {\it{mixed graph}} (with no loop and no multiple edge) is a couple $\Gamma = (V,E)$, where $V$ and $E$ are a non-empty finite set and a subset of $(V \times V) \setminus \{(v,v) \, : \, v \in V\}$, respectively. If $\Gamma$ fulfills: $(v_1,v_2) \in E$ if and only if $(v_2,v_1) \in E$ for $v_1 \not= v_2 $ (resp., if $(v_1,v_2) \in E$, then $(v_2,v_1) \not \in E$ for $v_1 \not= v_2 $), then $\Gamma$ is an {\it{undirected graph}} (resp., a {\it{directed graph}}). Writing $V=\{v_1, \dots, v_{|V|}\}$, we consider the {\it{Hermitian adjacency matrix}} $A(\Gamma)$ of $\Gamma$, which is the $|V| \times |V|$-matrix whose $(k,j)$-entry $a_{kj}$ is defined by
$$a_{kj} = \left \{
\begin{array}{ll}
1 & \mbox{if} \, \, \, (v_k, v_j) \in E \, \, \, {\rm{and}} \, \, \, (v_j, v_k) \in E, \\
{\bf i} & \mbox{if} \, \, \, (v_k, v_j) \in E \, \, \, {\rm{and}} \, \, \, (v_j, v_k) \not \in E, \\
-{\bf i} & \mbox{if} \, \, \, (v_k, v_j) \not \in E \, \, \, {\rm{and}} \, \, \, (v_j, v_k) \in E, \\
0 & \mbox{otherwise}
\end{array}
\right.
$$
for $1 \leq k,j \leq |V|$, where ${\bf i} = \sqrt{-1}$. If $\Gamma$ is undirected, then $A(\Gamma)$ is the ``usual" adjacency matrix of $\Gamma$. The mixed graph $\Gamma$ is {\it{integral}} if the eigenvalues of $A(\Gamma)$, which are always real numbers, are in $\mathbb{Z}$.

Given a finite group $G$ and a subset $S$ of $G \setminus \{1\}$, the {\it{mixed Cayley graph}} ${\rm{Cay}}(G,S)$ over $G$ with respect to $S$ is the mixed graph $(G, \{(g,h) \in G^2 : g^{-1} h \in S\})$. Setting $T = \{s \in S \, : \, s^{-1} \not \in S\},$ the mixed Cayley graph ${\rm{Cay}}(G,S)$ is undirected (resp., directed) if and only if $T = \emptyset$ (resp., $T = S$), in which case we say that ${\rm{Cay}}(G,S)$ is an {\it{undirected Cayley graph}} (resp., a {\it{directed Cayley graph}}). 

In the next section, we shall be interested in the integrality of mixed Cayley graphs, and hence, shall make use of the following result, which follows from Babai's work \cite{Bab79}:

\begin{lemma} \label{lemma:babai}
Let $G$ be a finite group, let $S \subseteq G \setminus \{1\}$, and let $T = \{s \in S \, : \, s^{-1} \not \in S\}$. Let $\rho_1, \dots, \rho_n$ be the inequivalent irreducible representations of $G$. Then the eigenvalues of the Hermitian adjacency matrix of the mixed Cayley graph ${\rm{Cay}}(G,S)$ are those of
$$\sum_{s \in S \setminus T} \rho_j(s) + {\bf i} \sum_{s \in T} \rho_j(s) - {\bf i} \sum_{s \in T} \rho_j(s^{-1}) \, \, \, {\rm{for}} \, \, \, j \in \{1, \dots, n\}.$$
\end{lemma}

\begin{proof}
Denote the regular representation of $G$ by $\rho^{{\rm reg}}$. By the proof of \cite[Theorem 3.1]{Bab79}, letting $\alpha : G \rightarrow \mathbb{C}$ denote the map defined by
$$\alpha(g) = \left \{
\begin{array}{ll}
1 & \mbox{if} \, \, \, g \in S \setminus T, \\
{\bf i} & \mbox{if} \, \, \, g \in T, \\
-{\bf i} & \mbox{if} \, \, \, g^{-1} \in T, \\
0 & \mbox{otherwise},
\end{array}
\right.$$
the Hermitian adjacency matrix of ${\rm Cay}(G,S)$ is
$$\sum_{g \in G}\alpha(g)\rho^{{\rm reg}}(g)=\sum_{s \in S \setminus T}\rho^{{\rm reg}}(s)+{\bf i} \sum_{s \in T}\rho^{{\rm reg}}(s)-{\bf i}\sum_{s \in T}\rho^{{\rm reg}}(s^{-1}).$$ 
Recall that $\rho^{{\rm reg}}$ is equivalent to $m_1 \rho_1 \oplus  \dots \oplus m_n \rho_n$, where $m_j$ is the dimension of $\rho_j$ for every $j \in \{1, \dots, n\}$. Therefore, we deduce that the eigenvalues of the Hermitian adjacency matrix of ${\rm Cay}(G,S)$ are those of
$$\sum_{s \in S \setminus T} \rho_j(s) + {\bf i} \sum_{s \in T} \rho_j(s) - {\bf i} \sum_{s \in T} \rho_j(s^{-1}) \, \, \, {\rm{for}} \, \, \, j \in \{1, \dots, n\},$$
as needed for the lemma.
\end{proof}

\section{Characterization of integrality} \label{sec:main}

\subsection{Main result} \label{ssec:main_1}

The main aim of this section is Theorem \ref{thm:main}, which characterizes integral mixed Cayley graphs over any given non-abelian finite group admitting an abelian subgroup of index 2. First, we observe that non-abelian finite groups $G$ admitting an abelian subgroup $A$ of index 2 (and order at least 3 without loss of generality) are exactly those admitting a presentation
\begin{equation} \label{pres}
G = \langle A, x \, | \, x^2 = y, xax^{-1} = f(a) \, \, (a \in A) \rangle,
\end{equation}
where $y$ is any given element of $A$ and where $f$ is any given automorphism of $A$ of order 2 and fixing $y$. Therefore, for this section, let $A$ be a finite abelian group of order at least $3$, let $f \in {\rm{Aut}}(A)$ be of order 2, and let $y \in A$ be such that $f(y)=y$. We define the finite group $G$ by the presentation given in \eqref{pres}. Note that $G$ is the disjoint union of $A$ and $xA$. From now on, we consider the subgroup $B =\{f(a)a^{-1} \, : \, a \in A\}$ of $A$.

\begin{theorem} \label{thm:main}
Let $S \subseteq G \setminus \{1\}$, and let $T = \{s \in S \, : s^{-1} \not \in S\}.$ Set $S \setminus T = S_1 \cup x S_2$ and $T = T_1 \cup x T_2$ with $S_1, S_2, T_1, T_2 \subseteq A$. Then ${\rm{Cay}}(G,S)$ is integral if and only if the next two conditions hold:

\vspace{0.5mm}

\noindent
{\rm{(1)}} $\rho(S_1) + \rho(x) \rho(S_2) - 2 {\rm{Im}}(\rho(T_1)) - 2 {\rm{Im}}(\rho(x) \rho(T_2)) \in \mathbb{Z}$ for every one-dimensional representation $\rho : G \rightarrow \mathbb{C}^*$ of $G$,

\vspace{0.5mm}

\noindent
{\rm{(2)}} for every one-dimensional representation $\pi : A \rightarrow \mathbb{C}^*$ of $A$ with $B \not \subseteq {\rm{ker}}(\pi)$, we have $\delta(\pi) \in \mathbb{Z}$ and there exists $\zeta \in \mathbb{Z}$ such that $\delta(\pi)^2 - 4 \epsilon(\pi) = \zeta^2$,
where 
\begin{equation} \label{eq:cinq1}
\left\{
\begin{array}{lll}
\alpha(\pi) =\pi(T_1) - \pi(T_1^{-1}), \\
\beta(\pi) = \pi(yf(T_2)) - \pi(T_2^{-1}), \\
\gamma(\pi) = \pi(f(T_1)) - \pi(f(T_1^{-1})), \\
\delta(\pi) = \pi(f(S_1)) + \pi(S_1) + {\bf i} (\alpha(\pi) + \gamma(\pi)),\end{array}
\right.
\end{equation}
 and $\epsilon(\pi)$ equals
\begin{equation} \label{eq:epsilon}
\pi(S_1) \pi(f(S_1)) - \pi(S_2) \pi(S_2^{-1}) + {\bf i}(\pi(S_1) \gamma(\pi) + \pi(f(S_1)) \alpha(\pi) - \beta(\pi)\pi(S_2) + \overline{\beta(\pi)}\pi(S_2^{-1})) \\
- \alpha(\pi) \gamma(\pi) - \beta(\pi) \overline{\beta(\pi)}.
\end{equation}
\end{theorem}

\subsection{Proof of Theorem \ref{thm:main}} \label{ssec:main_2}

To prove our main result, we shall make use of Lemma \ref{lemma:babai}, and so prove that the eigenvalues of 
$$\sum_{s \in S \setminus T} \rho(s) + {\bf i} \sum_{s \in T} \rho(s) - {\bf i} \sum_{s \in T} \rho(s^{-1}),$$
for $\rho$ ranging over the irreducible representations of $G$, are in $\mathbb{Z}$ if and only if (1) and (2) in the statement of the theorem hold. To that end, we first determine all irreducible representations of $G$, up to equivalence.

\subsubsection{Irreducible representations of $G$} \label{sssec:main_1}

The next classification, for which we could not find an explicit reference in the literature, generalizes and unifies those for genera\-lized dihedral groups and generalized dicyclic groups from \cite{HL21, BL22}.

\vspace{2mm}

\noindent
(1) {\it{One-dimensional representations.}} Let $\rho : G \rightarrow \mathbb{C}^*$ be a one-dimensional representation of $G$. Then $\rho(a)=\rho(x)\rho(a)\rho(x)^{-1}=\rho(xax^{-1})=\rho(f(a))$ and so $\rho(f(a) a^{-1})=1$ for $a \in A$, i.e., $\rho$ is trivial on $B$. Hence, every one-dimensional representation of $G$ gives rise to a one-dimensional representation of $A/B$. 

Conversely, let $\pi:A/B \rightarrow \mathbb{C}^*$ be a one-dimensional representation of $A/B$. Then there are exactly two inequivalent one-dimensional representations $\rho_{1}$ and $\rho_2$ of $G$ giving rise to $\pi$ and, denoting the reduction modulo $B$ of $a \in A$ by $\overline{a}$ and fixing a squareroot $\sqrt{\pi(\overline{y})}$ of $\pi(\overline{y})$ in $\mathbb{C}$, they are actually given by 

\noindent
$\bullet$ $\rho_1(x) = \sqrt{\pi(\overline{y})}$ and $\rho_1(a)=\pi(\overline{a})$ for $a \in A$, 

\noindent
$\bullet$ $\rho_2(x) = - \sqrt{\pi(\overline{y})}$ and $\rho_2(a)=\pi(\overline{a})$ for $a \in A$.

In particular, $G$ has exactly $2(A:B)$ inequivalent one-dimensional representations. Note that $(A:B)$ is the index of $B$ in $A$.

\vspace{2mm}

\noindent
(2) {\it{Two-dimensional representations.}} Let $\pi : A \rightarrow \mathbb{C}^*$ be a one-dimensional re\-pre\-sen\-tation of $A$ with $B \not \subseteq {\rm{ker}}(\pi)$. Let $R_\pi : G \rightarrow {\rm{GL}}_2(\mathbb{C})$ be the map defined by
\begin{equation}\label{equmatrepinds}
R_{\pi}(a)= \begin{pmatrix} \pi(a) & 0 \\ 0 & \pi(f(a))  \end{pmatrix} \quad {\rm{and}} \, \, \, R_\pi(xa) = \begin{pmatrix} 0 & \pi(yf(a)) \\ \pi(a) & 0 \end{pmatrix} \, \, \, (a \in A).
\end{equation}
As $f$ has order 2 and fixes $y$, the map $R_\pi$ is a representation of $G$. In fact, $R_\pi$ is the re\-pre\-sen\-ta\-tion $\mathrm{Ind}_{A}^{G}(\pi)$ of $G$ induced by $\pi$. 

\begin{lemma} \label{lem:irreducible}
The representation $R_\pi$ is irreducible.
\end{lemma}

\begin{proof}
As recalled in \S\ref{ssec:prelim_1}, it suffices to show 
\begin{equation} \label{eq:0}
(\chi_{R_\pi} | \chi_{R_\pi})=1,
\end{equation}
where $(\cdot | \cdot)$ is defined in \eqref{eq:scalar}. By \eqref{equmatrepinds}, we have
\begin{equation} \label{eq:1}
(\chi_{R_\pi} | \chi_{R_\pi}) = \frac{1}{2|A|} \sum_{a \in A} (\pi(a) + \pi(f(a)))(\pi(a^{-1})+\pi(f(a)^{-1}))  = 1 + \frac{1}{|A|} \sum_{a\in A} \pi(f(a) a^{-1}).
\end{equation}
Now, since $B \not \subseteq {\rm{ker}}(\pi)$, there exists $b \in A$ with $\pi(f(b) b^{-1}) \not=1$. Since
$$\sum_{a\in A} \pi(f(a) a^{-1}) = \sum_{a\in A} \pi(f(ab) (ab)^{-1}) = \pi(f(b) b^{-1}) \sum_{a\in A} \pi(f(a) a^{-1}),$$
we get
\begin{equation} \label{eq:2}
\sum_{a\in A} \pi(f(a) a^{-1})=0.
\end{equation}
It then remains to combine \eqref{eq:1} and \eqref{eq:2} to get \eqref{eq:0}, as needed for the lemma.
\end{proof}

\begin{lemma} \label{lemma:conj}
For $i\in \{1,2\}$, let $\pi_i : A \rightarrow \mathbb{C}^*$ be a one-dimensional representation of $A$ with $B \not \subseteq {\rm{ker}}(\pi_i)$. Then $R_{\pi_1}$ and $R_{\pi_2}$ are equivalent if and only if $\pi_2= \pi_1$ or $\pi_2 = \pi_1 \circ f$.
\end{lemma}

\begin{proof}
First, assume $\pi_2= \pi_1$ or $\pi_2 = \pi_1 \circ f$. Then $\chi_{R_{\pi_1}} = \chi_{R_{\pi_2}}$, and hence, the representations $R_{\pi_1}$ and $R_{\pi_2}$ are equivalent. 

Conversely, assume $R_{\pi_1}$ and $R_{\pi_2}$ are equivalent. Then $\chi_{R_{\pi_1}} = \chi_{R_{\pi_2}}$, and so 
\begin{equation} \label{eq:pi}
\pi_1(a)+\pi_1(f(a))=\pi_2(a)+\pi_2(f(a)) \, \, \, (a \in A).
\end{equation} 
Fix $a \in A$. As $\pi_1(a)$ and $\pi_2(a)$ are roots of unity, \eqref{eq:pi} gives 
$$|1+\pi_1(f(a) a^{-1})|=|1+\pi_2(f(a)a^{-1})|,$$ 
and so 
\begin{equation} \label{eq:pi2}
\pi_1(f(a)a^{-1})=\overline{\pi_2(f(a) a^{-1})} \, \, \, {\rm{or}} \, \, \, \pi_1(f(a) a^{-1})=\pi_2(f(a) a^{-1}).
\end{equation}

First, assume $\pi_1(f(a) a^{-1}) \not= -1$. If $\pi_1(f(a)a^{-1})=\overline{\pi_2(f(a) a^{-1})}$, then, using \eqref{eq:pi}, we get $$\pi_1(a)+\pi_1(f(a))= \pi_2(a) + \pi_2(f(a)) = \pi_2(a)+\pi_2(a)\overline{\pi_1(f(a)a^{-1})}.$$ 
We then have
$$\pi_1(a) = \pi_2(a) \frac{1 + \overline{\pi_1(f(a)a^{-1})}}{1 + \pi_1(f(a)a^{-1})} = \pi_2(a) \frac{1}{\pi_1(f(a)a^{-1})},$$
and so $\pi_2(a) =  \pi_1(f(a))$. If $\pi_1(f(a) a^{-1}) \not=\overline{\pi_2(f(a) a^{-1})}$, then, by \eqref{eq:pi2}, we have $\pi_1(f(a) a^{-1}) = \pi_2(f(a) a^{-1})$ and, using \eqref{eq:pi} once more, we get 
$$\pi_1(a)+\pi_1(f(a)) = \pi_2(a)+\pi_2(f(a)) =\pi_2(a)+\pi_2(a)\pi_1(f(a)a^{-1}).$$ 
In particular, $\pi_2(a)=\pi_1(a)$. 

Now, assume $\pi_1(f(a) a^{-1}) = -1$. Then $\pi_1(f(a^2) a^{-2}) = 1$ and, from the previous paragraph (applied with $a^2$ instead of $a$), we get either $\pi_2(a^2) = \pi_1(a^2)$ or $\pi_2(a^2) = \pi_1(f(a^2))$. In particular, at least one of the following four conditions holds:

\noindent
$\bullet$ $\pi_2(a) = \pi_1(a)$,

\noindent
$\bullet$ $\pi_2(a) = - \pi_1(a)$,

\noindent
$\bullet$ $\pi_2(a) = \pi_1(f(a))$,

\noindent
$\bullet$ $\pi_2(a) = - \pi_1(f(a))$.

\noindent
First, assume $\pi_2(a) = - \pi_1(f(a))$. Then $\pi_2(a) = - \pi_1(a) \pi_1(f(a)a^{-1}) = \pi_1(a)$, since we have assumed $\pi_1(f(a)a^{-1})= -1$ at the beginning of the paragraph. Now, assume $\pi_2(a) = - \pi_1(a)$. Then $\pi_1(f(a)) = \pi_1(f(a)a^{-1}) \pi_1(a) = - \pi_1(a) = \pi_2(a).$

Hence, from the previous two paragraphs, for $a \in A$, either $\pi_2(a)=\pi_1(a)$ or $\pi_2(a)=\pi_1(f(a))$. We now consider a case distinction. If $\pi_2(a)=\pi_1(a)$ for all $a \in A$, then $\pi_2=\pi_1$. Else, there is $a \in A$ with $\pi_1(a)\neq \pi_1(f(a))$ and $\pi_2(a)=\pi_1(f(a))$. In this case we claim that $\pi_2 = \pi_1 \circ f$. Indeed, assume there is $b \in A$ with $\pi_2(b) \neq \pi_1(f(b))$, so $\pi_2(b)=\pi_1(b)$. If $\pi_2(ab)=\pi_1(ab)$, then  $$\pi_1(a)\pi_1(b)=\pi_1(ab)=\pi_2(ab)=\pi_2(a)\pi_2(b)=\pi_1(f(a)) \pi_1(b),$$ 
so $\pi_1(a)=\pi_1(f(a))$, a contradiction. If $\pi_2(ab)=\pi_1(f(ab))$, a similar computation leads to $\pi_2(a)\pi_2(b)=\pi_2(a) \pi_1(f(b))$ and so $\pi_2(b)=\pi_1(f(b))$, which cannot happen either. 
\end{proof}

As there are exactly $|A|-(A:B)$ inequivalent one-dimensional representations $\pi : A \rightarrow \mathbb{C}^*$ of $A$ with $B \not \subseteq {\rm{ker}}(\pi)$, we then deduce from Lemmas \ref{lem:irreducible} and \ref{lemma:conj} that $G$ has at least $(|A|-(A:B))/2$ inequivalent two-dimensional irreducible representations.

\vspace{2mm}

\noindent
(3) {\it{Conclusion.}} Since 
$$1^2(2(A:B))+2^2\left( \frac{|A|-(A:B)}{2} \right)=2|A|=|G|,$$ 
we get from \eqref{eq:squares} that $G$ has exactly $(|A|-(A:B))/2$ inequivalent two-dimensional irreducible representations, and that $G$ has no further irreducible representation, up to equivalence. Therefore, we have proved the following result:

\begin{theorem} \label{thm:classification}
The group $G$ has exactly $2(A:B) + (|A| - (A:B))/2$ inequivalent irreducible representations 
$$\rho_1, \dots, \rho_{(A:B)}, \rho_{(A:B)+1}, \dots, \rho_{2(A:B)}, \rho_{2(A:B) + 1}, \dots, \rho_{2(A:B) + (|A| - (A:B))/2},$$
which are as follows.

\vspace{1mm}

\noindent
{\rm{(1)}} Let $\pi_1, \dots, \pi_{(A:B)}$ be the inequivalent irreducible representations of $A/B$. Then, given $j \in \{1, \dots, (A:B)\}$, we have

\vspace{0.5mm}

\noindent
$\bullet$ $\rho_j(x) = \sqrt{\pi_j(\overline{y})}$ and $\rho_{j}(a)=\pi_j(\overline{a})$ for $a \in A$,

\vspace{0.5mm}

\noindent
$\bullet$ $\rho_{j + (A:B)}(x) = - \sqrt{\pi_j(\overline{y})}$ and $\rho_{j + (A:B)}(a)=\pi_j(\overline{a})$ for $a \in A$,

\vspace{0.5mm}

\noindent
where $\overline{a}$ and $\sqrt{\pi(\overline{y})}$ are the reductions modulo $B$ of $a \in A$ and a given squareroot of $\pi(\overline{y})$ in $\mathbb{C}$, respectively.

\vspace{1mm}

\noindent
{\rm{(2)}} Let $\psi_{1}, \dots, \psi_{|A| - (A:B)}$ be the inequivalent irreducible representations of $A$ whose kernels do not contain $B$. Assume $\psi_{j + (|A| - (A:B))/2} = \psi_j \circ f$ for $j \in \{1, \dots, (|A| - (A:B))/2\}$. Then, given $j \in \{1, \dots, (|A| - (A:B))/2\}$, we have
$$\rho_{2(A:B) + j}(a)= \begin{pmatrix} \psi_j(a) & 0 \\ 0 & \psi_j(f(a))  \end{pmatrix} \quad {\rm{and}} \, \, \, \rho_{2(A:B) + j}(xa) = \begin{pmatrix} 0 & \psi_j(yf(a)) \\ \psi_j(a) & 0 \end{pmatrix} \, \, \, (a \in A).$$
\end{theorem}

\subsubsection{Proof of Theorem \ref{thm:main}} \label{sssec:main_2}

We now proceed to the proof of our main result.

\vspace{2mm}

\noindent
(a) {\it{One-dimensional representations.}} Given a one-dimensional representation $\rho : G \rightarrow \mathbb{C}^*$ of $G$, we have
$$\rho(S \setminus T) = \rho(S_1) + \rho(x) \rho(S_2),$$ 
$${\bf i} \rho(T_1) - {\bf i} \rho(T_1^{-1}) = {\bf i} \rho(T_1) + \overline{{\bf i} \rho(T_1)} = 2 {\rm{Re}}({\bf i} \rho(T_1)) = -2 {\rm{Im}}(\rho(T_1)),$$
$${\bf i} \rho(xT_2) - {\bf i} \rho((xT_2)^{-1}) = {\bf i} \rho(xT_2) + \overline{{\bf i} \rho(xT_2)} = - 2 {\rm{Im}}(\rho(x) \rho(T_2)).$$
Hence, $\rho(S \setminus T) + {\bf i} \rho(T) - {\bf i} \rho(T^{-1}) \in \mathbb{Z}$ for every one-dimensional representation $\rho : G \rightarrow \mathbb{C}^*$ of $G$ if and only if (1) holds.

\vspace{2mm}

\noindent
(b) {\it{Two-dimensional representations.}} Let $\pi : A \rightarrow \mathbb{C}^*$ be a one-dimensional re\-pre\-sen\-tation of $A$ with $B \not \subseteq {\rm{ker}}(\pi)$, and let $R_\pi : G \rightarrow {\rm{GL}}_2(\mathbb{C})$ be as in \eqref{equmatrepinds}. First, we have
$$\begin{array}{lllll}
\displaystyle{\sum_{a \in S \setminus T}R_\pi(a)}&=&\displaystyle{ \sum_{a \in S_1}R_\pi(a)+\sum_{a \in S_2}R_\pi(xa)}
                  &=& \displaystyle{ \sum_{a \in S_1} \begin{pmatrix} \pi(a) & 0 \\ 0 & \pi(f(a)) \end{pmatrix} + \sum_{a \in S_2} \begin{pmatrix} 0 & \pi(yf(a)) \\ \pi(a) & 0 \end{pmatrix} }\\
                  &&&= & \begin{pmatrix} \pi(S_1) & 0 \\
                  0 & \pi(f(S_1))\end{pmatrix} + \begin{pmatrix} 0 & \pi(yf(S_2)) \\ \pi(S_2) & 0 \end{pmatrix} \\
                  &&&= & \begin{pmatrix} \pi(S_1) & \pi(yf(S_2)) \\
                  \pi(S_2) & \pi(f(S_1))\end{pmatrix} 
\end{array}    $$
and, similarly, 
$$\sum_{a \in T}R_\pi(a) = \begin{pmatrix} \pi(T_1) & \pi(yf(T_2)) \\
                  \pi(T_2) & \pi(f(T_1))\end{pmatrix}.$$   
Moreover, given $a \in A$, we have $(xa)^{-1} = a^{-1} x^{-1} = x^{-1} f(a^{-1}) = x y^{-1} f(a^{-1}) \in xA$. Combined with $(S \setminus T)^{-1} = S \setminus T$, this yields $xS_2 = (xS_2)^{-1} = x y^{-1} f(S_2^{-1})$, i.e., $f(S_2^{-1}) = y S_2$. Since $f$ has order 2 and fixes $y$, we eventually get $S_2^{-1} = yf(S_2)$. In particular, we have
$$ \sum_{a \in S \setminus T}R_\pi(a) = \begin{pmatrix} \pi(S_1) & \pi(S_2^{-1}) \\
                  \pi(S_2) & \pi(f(S_1))\end{pmatrix}.$$
Furthermore, using once more $(xa)^{-1} = x^{-1} f(a^{-1})$ for $a \in A$, we have
\begin{align*}
\sum_{a \in T}R_\pi(a^{-1})&=\sum_{a \in T_1}R_\pi(a^{-1})+\sum_{a \in T_2}R_\pi(x^{-1}f(a^{-1}))\\
                  &=\sum_{a \in T_1} \begin{pmatrix} \pi(a^{-1}) & 0 \\ 0 & \pi(f(a^{-1})) \end{pmatrix} + \sum_{a \in T_2} \begin{pmatrix} 0 & 1 \\ \pi(y^{-1}) & 0 \end{pmatrix} \begin{pmatrix} \pi(f(a^{-1})) & 0 \\ 0 & \pi(f(f(a^{-1}))) \end{pmatrix}\\
                  &= \begin{pmatrix} \pi(T_1^{-1}) & 0 \\
                  0 & \pi(f(T_1^{-1}))\end{pmatrix} + \sum_{a \in T_2} \begin{pmatrix} 0 & 1 \\ \pi(y^{-1}) & 0 \end{pmatrix} \begin{pmatrix} \pi(f(a^{-1})) & 0 \\ 0 & \pi(a^{-1}) \end{pmatrix}\\
                  &= \begin{pmatrix} \pi(T_1^{-1}) & 0 \\
                  0 & \pi(f(T_1^{-1}))\end{pmatrix} + \sum_{a \in T_2} \begin{pmatrix} 0 & \pi(a^{-1}) \\ \pi(y^{-1}f(a^{-1})) & 0 \end{pmatrix}  \\
& = \begin{pmatrix} \pi(T_1^{-1}) & 0 \\
                  0 & \pi(f(T_1^{-1}))\end{pmatrix} + \begin{pmatrix} 0 & \pi(T_2^{-1}) \\
                  \pi(y^{-1} f(T_2^{-1})) & 0 \end{pmatrix} \\
                  &=\begin{pmatrix} \pi(T_1^{-1}) & \pi(T_2^{-1}) \\
                  \pi(y^{-1} f(T_2^{-1})) & \pi(f(T_1^{-1}))\end{pmatrix}.
\end{align*}
Hence, 
\begin{equation} \label{eq:rep}
\sum_{a \in S \setminus T} R_\pi(a) + {\bf i} \sum_{a \in T} R_\pi(a) - {\bf i} \sum_{a \in T} R_\pi(a^{-1})
\end{equation}
equals
$$\begin{pmatrix} \pi(S_1) + {\bf i} (\pi(T_1) - \pi(T_1^{-1})) & \pi(S_2^{-1}) + {\bf i}(\pi(yf(T_2)) - \pi(T_2^{-1})) \\ \pi(S_2) + {\bf i}(\pi(T_2) - \pi(y^{-1}f(T_2^{-1}))) & \pi(f(S_1)) + {\bf i}(\pi(f(T_1)) - \pi(f(T_1^{-1}))) \end{pmatrix}  \\.$$
Using \eqref{eq:cinq1} and \eqref{eq:epsilon}, we get that \eqref{eq:rep} equals
$$\begin{pmatrix} \pi(S_1) + {\bf i} \alpha(\pi) & \pi(S_2^{-1}) +{\bf i} \beta(\pi) \\ \pi(S_2) - {\bf i} \overline{\beta(\pi)} & \pi(f(S_1)) + {\bf i} \gamma(\pi) \end{pmatrix}  \\,$$
whose characteristic polynomial equals $X^2 - \delta(\pi) X + \epsilon(\pi).$ Hence, the eigenvalues of  \eqref{eq:rep} are 
$$\lambda_{1}= \frac{\delta(\pi) - \sqrt{\delta(\pi)^2 - 4 \epsilon(\pi)}}{2} \, \, \, {\rm{and}} \, \, \, \lambda_{2}=\frac{\delta(\pi) + \sqrt{\delta(\pi)^2 - 4 \epsilon(\pi)}}{2}.$$
If $\lambda_1, \lambda_2$ are in $\mathbb{Z}$, then $\delta(\pi)=\lambda_1+\lambda_2$ is in $\mathbb{Z}$ and $\sqrt{\delta(\pi)^2 - 4 \epsilon(\pi)} = \lambda_2 - \lambda_1 \in \mathbb{Z}$. Conversely, assume $\delta(\pi) \in \mathbb{Z}$ and $\sqrt{\delta(\pi)^2 - 4 \epsilon(\pi)} \in \mathbb{Z}$. Then $4 \epsilon(\pi) = \delta(\pi)^2 - (\sqrt{\delta(\pi)^2 - 4 \epsilon(\pi)})^2 \in \mathbb{Z}$. In particular, we have $\epsilon(\pi) \in \mathbb{Q}$. Since $\epsilon(\pi)$ is also integral over $\mathbb{Z}$, we actually have $\epsilon(\pi) \in \mathbb{Z}$. Hence, $\delta(\pi)$ and $\sqrt{\delta(\pi)^2 - 4 \epsilon(\pi)}$ have the same parity, and we get $\lambda_1, \lambda_2 \in \mathbb{Z}$. 

In view of the classification from Theorem \ref{thm:classification}, we then get that the eigenvalues of 
$$\sum_{s \in S \setminus T} \rho(s) + {\bf i} \sum_{s \in T} \rho(s) - {\bf i} \sum_{s \in T} \rho(s^{-1})$$
are in $\mathbb{Z}$ for every irreducible two-dimensional representation $\rho$ of $G$ if and only if (2) holds.

\subsection{Corollaries} \label{ssec:main_3}

The statement of Theorem \ref{thm:main} is technical because we consider mixed Cayley graphs over arbitrary non-abelian finite groups admitting an abelian subgroup of index 2. As illustrations of our main result, we now consider two special cases of interest.

We first restrict to the undirected situation, in which case Theorem \ref{thm:main} takes the next simpler form:

\begin{corollary} \label{coro:undirected}
Let $S$ be a subset of $G$ with $1 \not \in S$ and $S^{-1}=S$. Set $S = S_1 \cup x S_2$ with $S_1, S_2 \in A$. Then the undirected Cayley graph ${\rm{Cay}}(G,S)$ is integral if and only if the next two conditions hold:

\vspace{0.5mm}

\noindent
{\rm{(1)}} $\rho(S_1) + \rho(x) \rho(S_2) \in \mathbb{Z}$ for every one-dimensional representation $\rho : G \rightarrow \mathbb{C}^*$ of $G$,

\vspace{0.5mm}

\noindent
{\rm{(2)}} for every one-dimensional representation $\pi : A \rightarrow \mathbb{C}^*$ with $B \not \subseteq {\rm{ker}}(\pi)$, we have $\pi(f(S_1)) + \pi(S_1) \in \mathbb{Z}$ and there exists $\zeta \in \mathbb{Z}$ such that 
$$(\pi(f(S_1)) - \pi(S_1))^2 + 4 \pi(S_2) \pi(S_2^{-1}) = \zeta^2.$$
\end{corollary}

The following corollary allows to generate integral undirected Cayley graphs over $G$:

\begin{corollary} \label{coro:simple}
 Let $S_1, S_2 \in \mathbb{B}(A)$ be such that $1 \not \in S_1$, $f(S_1) = S_1$ and $f(S_2) = y^{-1}S_2$. Then the un\-di\-rected Cayley graph ${\rm{Cay}}(G, S_1 \cup xS_2)$ is integral.
\end{corollary}

\begin{proof}
First, the undirected Cayley graph ${\rm{Cay}}(G, S_1 \cup xS_2)$ is well-defined. Indeed, since $S_1 \in \mathbb{B}(A)$, Lem\-ma \ref{lemma:ap_0} yields that $S_1^{-1}$ and $S_1$ coincide. Similarly, we have $S_2^{-1} = S_2$ and, as $f(S_2) = y^{-1}S_2$, we get
\begin{equation} \label{eq:inverse}
(xS_2)^{-1} = S_2^{-1} x^{-1} = x^{-1} f(S_2^{-1}) = x y^{-1} f(S_2^{-1}) = xy^{-1} y S_2^{-1} = xS_2^{-1} = xS_2,
\end{equation}
thus yielding $(S_1 \cup xS_2)^{-1} = S_1 \cup xS_2$, as needed.

Now, to get Corollary \ref{coro:simple}, it suffices to show that (1) and (2) in Corollary \ref{coro:undirected} hold. First, let $\rho : G \rightarrow \mathbb{C}^*$ be a one-dimensional representation of $G$. To get (1), we have to show $\rho(S_1) + \rho(x) \rho(S_2) \in \mathbb{Z}$. To that end, note that, as $S_j \in \mathbb{B}(A)$ for $j \in \{1,2\}$, we have $\rho(S_j) \in \mathbb{Z}$ by Lemma \ref{lemma:ap}. In particular, if $\rho(S_2) = 0$, then $\rho(S_1) + \rho(x) \rho(S_2) \in \mathbb{Z}$. Therefore, assume $\rho(S_2) \not=0$. Then, by \eqref{eq:inverse}, we have $\rho(x) \rho(S_2) = \rho((xS_2)^{-1}) = \overline{\rho(x) \rho(S_2)}$, i.e., $\rho(x) \rho(S_2) \in \mathbb{R}$. Since $\rho(S_2) \in \mathbb{Z} \setminus \{0\}$, we get $\rho(x) \in \mathbb{R}$. But $\rho(x)$ is also a root of unity, and hence, $\rho(x) \in \{-1, 1\}$, and so $\rho(x) \rho(S_2) \in \mathbb{Z}$. In particular, (1) holds.

Now, let $\pi : A \rightarrow \mathbb{C}^*$ be a one-dimensional representation of $A$ with $B \not \subseteq {\rm{ker}}(\pi)$. Since $f(S_1) = S_1$, we have $\pi(f(S_1)) + \pi(S_1) = 2 \pi(S_1)$ and, as $S_1 \in \mathbb{B}(A)$, we have $\pi(S_1) \in \mathbb{Z}$ by Lemma \ref{lemma:ap}. Therefore, $\pi(f(S_1)) + \pi(S_1) \in \mathbb{Z}$. Moreover, since $S_2^{-1} = S_2$, we have
$$(\pi(f(S_1)) - \pi(S_1))^2 + 4 \pi(S_2) \pi(S_2^{-1}) = 4 \pi(S_2)^2$$
and, as $S_2 \in \mathbb{B}(A)$, we have $\pi(S_2) \in \mathbb{Z}$ by Lemma \ref{lemma:ap}. Hence, (2) holds.
\end{proof}

\begin{example} \label{ex:1}
Assume that $A$ has exponent $e$, where $e \geq 3$, and that there is an integer $s$ of order 2 modulo $e$ such that, for every $a \in A$, we have $f(a) = a^s$. Then, in the statement of Corollary \ref{coro:simple}, the assumption ``$f(S_1) = S_1$" is redundant and the one ``$f(S_2) = y^{-1} S_2$" reduces to ``$yS_2 = S_2$". Indeed, from our extra assumption on $f$, we have $f(S_j) = S_j^s$ for $j\in\{ 1,2\}$ and, since $s$ is coprime\footnote{In fact, since $s^2 \equiv 1 \pmod{e}$, we have ${\rm{gcd}}(s,e)=1$.} to every element order in $A$, we may then apply Lemma \ref{lemma:ap_0} to get $f(S_j) = S_j^s = S_j$ for $j \in \{ 1,2\}$.

In particular, if $A$ has exponent $e$, where $e \geq 3$, if there is an integer $s$ of order 2 modulo $e$ such that, for every $a \in A$, we have $f(a) = a^s$, and if $y=1$, then every undirected Cayley graph ${\rm{Cay}}(G, S_1 \cup xS_2)$, with $1 \not \in S_1$ and $S_1, S_2 \in \mathbb{B}(A)$, is integral. In addition to generalized dihedral groups (for which $s=-1$), the latter applies to, e.g., quasi-dihedral groups, which assume $A=\mathbb{Z}/2^{n-1}\mathbb{Z}$ ($n \geq 4$) and $s=2^{n-2}-1$, and modular groups, which assume $A=\mathbb{Z}/2^{n-1}\mathbb{Z}$ ($n \geq 4$) and $s=2^{n-2}+1$.
\end{example}

We now consider the case where $f(a) = a^{-1}$ for $a \in A$, in which case $B = A^2$ and $y^2=1$ (and $A$ has exponent at least 3). If $y=1$, then $G$ is the generalized dihedral group ${\rm{Dih}}(A)$, which is given by the following presentation:
$${\rm{Dih}}(A) = \langle A, x \, | \, x^2 = 1, xax^{-1} = a^{-1} \, (a \in A) \rangle.$$
If $y \not=1$ (and so $|A|$ is even), then $G$ is the generalized dicyclic group ${\rm{Dic}}(A,y)$, which is given by the following presentation:
$${\rm{Dic}}(A,y) = \langle A,x \, | \, x^2=y, xax^{-1}=a^{-1} \, (a \in A) \rangle.$$ 

The next corollary applies to both ${\rm{Dih}}(A)$ and ${\rm{Dic}}(A,y)$:

\begin{corollary} \label{coro:s=-1}
Assume $f(a) = a^{-1}$ for every $a \in A$. Let $S$ be a subset of $G$ with $1 \not \in S$, and let $T = \{s \in S \, : s^{-1} \not \in S\}.$ Set $S \setminus T = S_1 \cup x S_2$ and $T = T_1 \cup x T_2$ with $S_1, S_2, T_1, T_2 \subseteq A$. Then the mixed Cayley graph ${\rm{Cay}}(G,S)$ is integral if and only if the next three conditions hold:

\vspace{0.5mm}

\noindent
{\rm{(1)}} $S_1 \in \mathbb{B}(A)$,

\vspace{0.5mm}

\noindent
{\rm{(2)}} for every one-dimensional representation $\pi : A \rightarrow \mathbb{C}^*$ of $A$ with $A^2 \not \subseteq {\rm{ker}}(\pi)$ and $y \in {\rm{ker}}(\pi)$, there exists $\alpha \in \mathbb{Z}$ such that $\pi(S_2) \pi(S_2^{-1}) - (\pi(T_1) - \pi(T_1^{-1}))^2 = \alpha^2, $

\vspace{0.5mm}

\noindent
{\rm{(3)}} for every one-dimensional representation $\pi : A \rightarrow \mathbb{C}^*$ of $A$ with $A^2 \not \subseteq {\rm{ker}}(\pi)$ and $y \not \in {\rm{ker}}(\pi)$, there exists $\alpha \in \mathbb{Z}$ such that
$ 4 \pi(T_2) \pi(T_2^{-1}) - (\pi(T_1) - \pi(T_1^{-1}))^2 = \alpha^2. $
\end{corollary}

\begin{proof}
By Theorem \ref{thm:main}, it suffices to prove the following four statements, in which we use the no\-ta\-tion from our main result:

\vspace{1mm}

\noindent
(A) We have $\rho(S_1) + \rho(x) \rho(S_2) - 2 {\rm{Im}}(\rho(T_1)) - 2 {\rm{Im}}(\rho(x) \rho(T_2)) \in \mathbb{Z}$ for every one-dimensional representation $\rho :G \rightarrow \mathbb{C}^*$ of $G$.

\vspace{1mm}

\noindent
(B) The next two conditions are equivalent:

\vspace{0.5mm}

(i) $S_1 \in {\mathbb{B}}(A)$,

\vspace{0.5mm}

(ii) $\delta(\pi) \in \mathbb{Z}$ for every one-dimensional representation $\pi : A \rightarrow \mathbb{C}^*$ of $A$ with $A^2 \not \subseteq {\rm{ker}}(\pi)$.

\vspace{1mm}

\noindent
(C) The next two conditions are equivalent:

\vspace{0.5mm}

(i) for every one-dimensional representation $\pi : A \rightarrow \mathbb{C}^*$ of $A$ with $A^2 \not \subseteq {\rm{ker}}(\pi)$ and $y \in {\rm{ker}}(\pi)$, there 

exists $\zeta \in \mathbb{Z}$ such that $\delta(\pi)^2 - 4 \epsilon(\pi) = \zeta^2$,

\vspace{0.5mm}

(ii) for every one-dimensional representation $\pi : A \rightarrow \mathbb{C}^*$ of $A$ such that $A^2 \not \subseteq {\rm{ker}}(\pi)$ and $y \in {\rm{ker}}(\pi)$, 

there exists $\alpha \in \mathbb{Z}$ satisfying $\pi(S_2) \pi(S_2^{-1}) - (\pi(T_1) - \pi(T_1^{-1}))^2 = \alpha^2.$

\vspace{1mm}

\noindent
(D) The next two conditions are equivalent:

\vspace{0.5mm}

(i) for every one-dimensional representation $\pi : A \rightarrow \mathbb{C}^*$ of $A$ with $A^2 \not \subseteq {\rm{ker}}(\pi)$ and $y \not \in {\rm{ker}}(\pi)$, there 

exists $\zeta \in \mathbb{Z}$ such that $\delta(\pi)^2 - 4 \epsilon(\pi) = \zeta^2$,

\vspace{0.5mm}

(ii) for every one-dimensional representation $\pi : A \rightarrow \mathbb{C}^*$ of $A$ such that $A^2 \not \subseteq {\rm{ker}}(\pi)$ and $y \not \in {\rm{ker}}(\pi)$, 

there exists $\alpha \in \mathbb{Z}$ satisfying $ 4 \pi(T_2) \pi(T_2^{-1}) - (\pi(T_1) - \pi(T_1^{-1}))^2 = \alpha^2. $

\vspace{1.5mm}

\noindent
For later use, we observe that, since $(S \setminus T)^{-1} = S \setminus T$, we must have 
\begin{equation} \label{eq:s1}
S_1^{-1} = S_1 \, \, \, {\rm{and}} \, \, \,  yS_2^j = S_2^j \, \, \, (j \in \{-1,1\}).
\end{equation}
In particular, for every one-dimensional representation $\pi : A \rightarrow \mathbb{C}^*$ of $A$ with $y \not \in {\rm{ker}}(\pi)$ (and so $\pi(y)=-1$), and for $j \in \{-1, 1\}$, we must have $\pi(S_2^j) = \pi(yS_2^j) = \pi(y) \pi(S_2^j) = - \pi(S_2^j)$, i.e.,
\begin{equation} \label{eq:zero}
\pi(S_2^j) = 0.
\end{equation}

\noindent
Proof of (A): Let $\rho : G \rightarrow \mathbb{C}^*$ be a one-dimensional representation of $G$. Since $f(a) = a^{-1}$ for $a \in A$, we have $A^2 \subseteq {\rm{ker}}(\rho)$, and hence, $\rho(a) \in \{-1, 1\}$ for $a \in A$. Moreover, since $x^2=y$ and $y^2=1$, we have $\rho(x) \in \{-1, 1, -{\bf i}, {\bf i}\}$. If $\rho(x) \in \{-1, 1\}$, then 
$$\rho(S_1) + \rho(x) \rho(S_2) - 2 {\rm{Im}}(\rho(T_1)) - 2 {\rm{Im}}(\rho(x) \rho(T_2)) \in \mathbb{Z}.$$ 
Therefore, assume $\rho(x) \in \{-{\bf i},{\bf i}\}$. Then $\rho(S_2) = 0$ by \eqref{eq:zero}. Since we also have $\rho(x) \rho(T_2) \in \mathbb{Z}[{\bf i}]$, we get 
$$\rho(S_1) + \rho(x) \rho(S_2) - 2 {\rm{Im}}(\rho(T_1)) - 2 {\rm{Im}}(\rho(x) \rho(T_2)) \in \mathbb{Z},$$ 
thus ending the proof of (A).

\vspace{1.5mm}

\noindent
Proof of (B): Let $\pi : A \rightarrow \mathbb{C}^*$ be a one-dimensional representation of $A$ with $A^2 \not \subseteq {\rm{ker}}(\pi)$. Using \eqref{eq:s1} and our assumption on $f$, we can rewrite \eqref{eq:cinq1} and \eqref{eq:epsilon} as follows:
\begin{equation} \label{eq:cinq2}
\left\{
\begin{array}{lll}
\alpha(\pi) = \pi(T_1) - \pi(T_1^{-1}), \\
\beta(\pi) = \pi(T_2^{-1})(\pi(y)-1), \\
\gamma(\pi)= \pi(T_1^{-1}) - \pi(T_1) = - \alpha(\pi), \\
\delta(\pi) = 2 \pi(S_1), \\
\epsilon(\pi) = \pi(S_1)^2 - \pi(S_2) \pi(S_2^{-1}) + {\bf i} (\overline{\beta(\pi)} \pi(S_2^{-1}) - \beta(\pi) \pi(S_2)) + \alpha(\pi)^2 - \beta(\pi) \overline{\beta(\pi)}. \\
\end{array}
\right.
\end{equation}
If $\delta(\pi) \in \mathbb{Z}$, then, by \eqref{eq:cinq2}, we have $\pi(S_1) \in \mathbb{Q}$. Since $\pi(S_1)$ is integral over $\mathbb{Z}$, we get that $\pi(S_1)$ is actually in $\mathbb{Z}$. On the other hand, we have $\pi(S_1) \in \mathbb{Z}$ for every one-dimensional representation $\pi : A \rightarrow \mathbb{C}^*$ of $A$ with $A^2 \subseteq {\rm{ker}}(\pi)$. Hence, if (ii) in (B) holds, then $\pi(S_1) \in \mathbb{Z}$ for every one-dimensional representation $\pi : A \rightarrow \mathbb{C}^*$ of $A$, that is, $S_1$ is integral. By Lemma \ref{lemma:ap}, we conclude that $S_1 \in \mathbb{B}(A)$. Conversely, if $S_1 \in \mathbb{B}(A)$, then $\pi(S_1) \in \mathbb{Z}$ for every one-dimensional representation $\pi : A \rightarrow \mathbb{C}^*$ of $A$, by Lemma \ref{lemma:ap}. In particular, for such $\pi$ also fulfilling $A^2 \not \subseteq {\rm{ker}}(\pi)$, we get that $\delta(\pi)$, which equals $2\pi(S_1)$ by \eqref{eq:cinq2}, is in $\mathbb{Z}$, thus ending the proof of (B).

\vspace{1.5mm}

\noindent
Proof of (C): Let $\pi : A \rightarrow \mathbb{C}^*$ be a one-dimensional representation of $A$ with $A^2 \not \subseteq {\rm{ker}}(\pi)$ and $y \in {\rm{ker}}(\pi)$. Then from \eqref{eq:cinq2}, we have 
$$\delta(\pi)^2 - 4 \epsilon(\pi) = 4 \pi(S_2) \pi(S_2^{-1}) - 4 (\pi(T_1) - \pi(T_1^{-1}))^2.$$
If there is $\alpha \in \mathbb{Z}$ with $\pi(S_2) \pi(S_2^{-1}) - (\pi(T_1) - \pi(T_1^{-1}))^2 = \alpha^2$, then it is clear that there is $\zeta \in \mathbb{Z}$ with $\delta(\pi)^2 - 4 \epsilon(\pi) = \zeta^2$. Conversely, assume there is $\zeta \in \mathbb{Z}$ with $\delta(\pi)^2 - 4 \epsilon(\pi) = \zeta^2$. Then 
$$\sqrt{\pi(S_2)\pi(S_2^{-1}) - (\pi(T_1) - \pi(T_1^{-1}))^2} \in \mathbb{Q}.$$ 
But this number is also integral over $\mathbb{Z}$, and hence, it is in $\mathbb{Z}$. This completes the proof of (C).

\vspace{1.5mm}

\noindent
Proof of (D): Let $\pi : A \rightarrow \mathbb{C}^*$ be a one-dimensional representation of $A$ with $A^2 \not \subseteq {\rm{ker}}(\pi)$ and $y \not \in {\rm{ker}}(\pi)$. Then, by \eqref{eq:zero} and \eqref{eq:cinq2}, we have
$$\delta(\pi)^2 - 4 \epsilon(\pi) = - 4 (\pi(T_1) - \pi(T_1^{-1}))^2 + 16 \pi(T_2) \pi(T_2^{-1}),$$
and it then remains to proceed as in the proof of (C) to get (D).
\end{proof}

If $G={\rm{Dih}}(A)$, then (3) is vacuous and the condition ``$y \in {\rm{ker}}(\pi)$" in (2) is automatic. Moreover, if $T = \emptyset$, one retrieves the characterization of integral undirected Cayley graphs over ${\rm{Dih}}(A)$ obtained in \cite[Theorem 3.1]{HL21}\footnote{Of course, via the statements of the proof of Corollary \ref{coro:s=-1}, we may also derive \cite[Theorem 3.1]{HL21} from Corollary \ref{coro:undirected}, and the same applies to the case of generalized dicyclic groups.}. On the other hand, the directed case is covered by the next statement:

\begin{corollary} \label{coro:1}
Let $A$ be a finite abelian group of exponent at least $3$ and let $S \subseteq A \setminus \{1\}$ be such that $s^{-1} \in S$ for no $s \in S$. \footnote{For a subset $S = T_1 \cup x T_2$ of ${\rm{Dih}}(A) \setminus \{1\}$ with $T_1, T_2 \subseteq A$ and $s^{-1} \in S$ for no $s \in S$, we necessarily have $T_2 = \emptyset$ (and hence, $S =T_1 \subseteq A$). Indeed, if there is $a \in T_2$, we have $(xa)^{-1} = a^{-1} x^{-1} = x^{-1} a = xa \in xT_2 \subseteq S$, which cannot hold.} Then the directed Cayley graph ${\rm{Cay}}({\rm{Dih}}(A),S)$ is integral if and only if $\pi(S) - \pi(S^{-1}) \in {\bf i}\mathbb{Z}$ for every one-dimensional representation $\pi : A \rightarrow \mathbb{C}^*$ of $A$ with $A^2 \not \subseteq {\rm{ker}}(\pi)$.
\end{corollary}

\begin{example} \label{ex:1.5}
As an illustration of our result, we now determine all integral directed Cayley graphs over the dihedral group with $2p$ elements, where $p$ is any given odd prime number. The following may be compared with \cite[Theorem 4.2]{LHH18}, which classifies all integral undirected Cayley graphs over such finite groups.

First, assume the exponent $e$ of $A$ is not divisible by 4. Then ${\rm{Cay}}({\rm{Dih}}(A),S)$ is integral if and only if $\pi(S) \in \mathbb{R}$ for every one-dimensional representation $\pi : A \rightarrow \mathbb{C}^*$ of $A$ with $A^2 \not \subseteq {\rm{ker}}(\pi)$. 

Indeed, fix such a representation $\pi$ of $A$. If $\pi(S) \in \mathbb{R}$, then $\pi(S) - \pi(S^{-1}) = 0 \in {\bf i} \mathbb{Z}$. Conversely, assume $\pi(S) - \pi(S^{-1}) \in {\bf i} \mathbb{Z}$. Then there exists $a \in \mathbb{Z}$ such that ${\bf i}a = \pi(S) - \pi(S^{-1})$. Note that $\pi(S) - \pi(S^{-1})$ is an element of the cyclotomic field $\mathbb{Q}(\zeta_e)$. Assume $a \not=0$, in which case we have ${\bf i} \in \mathbb{Q}(\zeta_e)$. If $e$ is odd, we get a contradiction (as $e$ and 4 are coprime) and, if $e$ is divisible by 2 (but not by 4), we get ${\bf i} \in \mathbb{Q}(\zeta_e) = \mathbb{Q}(\zeta_{e/2})$, which cannot hold either. Hence, $a=0$, i.e., $\pi(S) = \pi(S^{-1}) \in \mathbb{R}$.

Now, assume that $A$ is cyclic of prime order $p$ with $p$ odd. Fix a non-empty subset $S$ of $A$ with $1 \not \in S$ and $s^{-1} \in S$ for no $s \in S$, and assume ${\rm{Cay}}({\rm{Dih}}(A), S)$ is integral. Choosing a root of unity $\zeta$ of order $p$ and a generator $a$ of $A$, and considering the one-dimensional representation $\pi$ of $A$ defined by $\pi(a) = \zeta$, we must have $\pi(S) \in \mathbb{R}$. Setting $S = \{a^{i_1}, \dots, a^{i_n}\}$ with $n \geq 1$ and $i_j \in \{1, \dots, p-1\}$ for $j\in \{1, \dots, n\}$, we then have
$$\zeta^{i_1}  + \cdots + \zeta^{i_n} = \zeta^{-i_1}  + \cdots + \zeta^{-i_n} = \zeta^{p-i_1}  + \cdots + \zeta^{p-i_n}.$$
But $\zeta, \dots, \zeta^{p-1}$ are linearly independent over $\mathbb{Q}$ and, since $1 \leq p-i_j \leq p-1$ for $j\in \{1, \dots, n\}$, we get $\{i_1, \dots, i_n\} = \{p-i_1, \dots, p-i_n\}$, which contradicts the assumption that $s^{-1} \in S$ for no $s \in S$. Hence, the directed Cayley graph ${\rm{Cay}}({\rm{Dih}}(A), S)$ is not integral if $S \not= \emptyset$, and it is easily seen that the converse holds.

On the other hand, if the exponent of $A$ is divisible by 4, Corollary \ref{coro:1} yields examples of integral directed Cayley graphs over ${\rm{Dih}}(A)$. For instance, one easily obtains that every directed Cayley graph over the dihedral group with 8 elements is integral.
\end{example}

On the other hand, if $G={\rm{Dic}}(A,y)$ and $T = \emptyset$ in Corollary \ref{coro:s=-1}, one retrieves the characterization of integral undirected Cayley graphs over generalized dicyclic groups obtained in \cite[Theorem 3.1]{BL22}. The counterpart for directed Cayley graphs is given by the following statement:

\begin{corollary} \label{coro:2}
Let $A$ be a finite abelian group with even order and exponent at least $3$, and let $y \in A$ be of order 2. Let $S \subseteq {\rm{Dic}}(A,y) \setminus \{1\}$ be such that $s^{-1} \in S$ for no $s \in S$. Set $S= T_1 \cup xT_2$ with $T_1,T_2 \subseteq A$. Then the directed Cayley graph ${\rm{Cay}}({\rm{Dic}}(A,y),S)$ is integral if and only if the following two conditions hold:

\vspace{0.5mm}

\noindent
{\rm{(a)}} $\pi(T_1) - \pi(T_1^{-1}) \in {\bf i}\mathbb{Z}$ for every one-dimensional representation $\pi : A \rightarrow \mathbb{C}^*$ of $A$ with $A^2 \not \subseteq {\rm{ker}}(\pi)$ and $y \in {\rm{ker}}(\pi),$
 
\vspace{0.5mm}

\noindent
{\rm{(b)}} for every one-dimensional representation $\pi : A \rightarrow \mathbb{C}^*$ of $A$ with $A^2 \not \subseteq {\rm{ker}}(\pi)$ and $y \not  \in {\rm{ker}}(\pi)$, there exists $\alpha \in \mathbb{Z}$ such that $4\pi(T_2) \pi(T_2^{-1}) - (\pi(T_1) - \pi(T_1^{-1}))^2 = \alpha^2.$
\end{corollary}

\begin{example} \label{ex:2}
As an illustration of our result, we determine all integral directed Cayley graphs over the dicyclic group with 8 elements. 

To that end, set $A = \mathbb{Z}/4\mathbb{Z}$ and let $a$ be a generator of $A$, in which case we necessarily have $a^2=y$. Let $S \subseteq {\rm{Dic}}(A,y) \setminus \{1\}$ be such that $s^{-1} \in S$ for no $s \in S$. Set $S= T_1 \cup xT_2$ with $T_1,T_2 \subseteq A$. Given $b \in A$, we have $(xb)^{-1} = x^{-1} b  = xyb$. Hence, as $s^{-1} \in S$ for no $s \in S$, we have 
$$T_1 \in \{\emptyset, \{a\}, \{a^3\}\} \quad {\rm{and}} \quad T_2 \in\{\emptyset, \{1\}, \{a\}, \{a^2\}, \{a^3\}, \{1,a\}, \{1,a^3\}, \{a, a^2\}, \{a^2, a^3\} \}.$$

Now, let $\pi : A \rightarrow \mathbb{C}^*$ be a one-dimensional representation of $A$ with $A^2 \not \subseteq {\rm{ker}}(\pi)$. Since $A^2 = \langle y \rangle$, we have $y \not \in {\rm{ker}}(\pi)$, and hence, Condition (a) in Corollary \ref{coro:2} is empty. We then have $\pi(y) = -1$, which yields $\pi(a) = \pm {\bf i}$. In particular, we have
$$
(\pi(T_1) - \pi (T_1^{-1}))^2 = \left\{
    \begin{array}{lll}
        0 & {\rm{if}} & T_1 = \emptyset \\
        -4 & {\rm{if}} & |T_1|=1
    \end{array}
\right.
\quad {\rm{and}} \quad
4 \pi(T_2) \pi (T_2^{-1}) = \left\{
    \begin{array}{lll}
        0 & {\rm{if}} & T_2 = \emptyset \\
        4 & {\rm{if}} & |T_2|=1 \\
 8 & {\rm{if}} & |T_2|=2 \\
    \end{array}
\right..
$$
In view of Condition (b) of Corollary \ref{coro:2}, we get that the directed Cayley graph ${\rm{Cay}}({\rm{Dic}}(A,y), S)$ is integral if and only if either $T_1 \in \{\{a\}, \{a^3\}\}$ and $T_2 = \emptyset$, or $T_1 = \emptyset$ and $T_2 \in\{\emptyset, \{1\}, \{a\}, \{a^2\}, \{a^3\}\}$.
\end{example}

\bibliography{Biblio2}
\bibliographystyle{alpha}

\end{document}